\documentclass[a4paper,11pt,reqno]{article}
\usepackage[utf8]{inputenc}
\usepackage{amsmath}
\usepackage{amssymb}
\usepackage{mathrsfs}
\usepackage{mathtools}
\usepackage{enumerate}
\usepackage{enumitem}
\usepackage{amsthm}
\usepackage{booktabs}
\usepackage{xcolor}
\usepackage{url,hyperref}
\usepackage{multirow}
\allowdisplaybreaks
\usepackage{caption}
\usepackage{blindtext}
\usepackage{algpseudocode,algorithm}
\usepackage{bold-extra}

\usepackage{setspace}
\newcommand{\email}[1]{\href{mailto:#1}{\texttt{#1}}}

\usepackage[margin=3.5cm]{geometry} 
\usepackage{tabularx}

\DeclareMathSymbol{:}{\mathord}{operators}{"3A}

\usepackage{float}

\newcommand{\F}{\mathbb{F}}

\renewcommand{\L}{\mathcal{L}}

\newcommand{\M}{\mathcal{M}}
\newcommand{\rk}{\operatorname{rk}}

\newcommand{\supp}{\operatorname{supp}}
\newcommand{\C}{\mathcal{C}}

\renewcommand{\S}{\mathcal{S}}

\renewcommand{\epsilon}{\varepsilon}
\newcommand{\NN}{\mathbb N}

\newcommand{\GL}{\operatorname{GL}}

\newcommand{\qbinom}[2]{\genfrac[]{0pt}{0}{#1}{#2}}

\usepackage{tabularx}

\makeatletter
\renewcommand\@makefnmark{}
\makeatother

\theoremstyle{plain} \numberwithin{equation}{section}
\newtheorem{theorem}{Theorem}[section]
\newtheorem{corollary}[theorem]{Corollary}

\newtheorem{lemma}[theorem]{Lemma}
\newtheorem{proposition}[theorem]{Proposition}

\theoremstyle{definition}
\newtheorem{definition}[theorem]{Definition}

\newtheorem{remark}[theorem]{Remark}
\newtheorem{example}[theorem]{Example}

\newtheorem{proposition/definition}[theorem]{Proposition/Definition}

\renewcommand{\phi}{\varphi}

\newif\ifcomment
\commenttrue

\title{\textbf{Almost Affine Vector Rank-Metric Codes}}

\author{Matteo Bonini\thanks{Department of Mathematical Sciences, Aalborg University, Aalborg, Denmark (\email{mabo@math.aau.dk})} \and Johan Vester Dinesen\thanks{Department of Mathematics and Systems Analysis, Aalto University, Espoo, Finland (\email{johan.v.dinesen@aalto.fi})}}
\date{}

\begin{document}

\footnotetext{\textbf{Funding:} M.~Bonini was supported by \textit{Blue Mathematics at AAU}, funded by Orient's Fond. J.V.~Dinesen was supported by the European Union MSCA Doctoral Networks (HORIZON-MSCA-2021-DN-01, Project 101072316).
}
\maketitle
\begin{abstract}
We define almost affine vector rank-metric codes as subsets $\C\subseteq \F_{q^m}^n$ whose canonical projections have cardinalities that are powers of $q^m$, and prove that they naturally induce $q$-matroids. We establish that the operations of puncturing and shortening correspond to restriction and contraction of the $q$-matroid, and show that the rank-weight and formal dual distance distributions are determined by the induced $q$-matroid. We briefly discuss applications to perfect $q$-matroid ports in linear network coding, and show that disconnected $q$-matroids need not induce disconnected ports. Finally, we show that certain Additive Generalized Twisted Gabidulin codes yield direct examples of strictly almost affine rank-metric codes, alongside a separate construction derived from proper finite semifields.
\end{abstract}

\section{Introduction}

Almost affine codes in the Hamming metric were first introduced in \cite{AlmostAffineCodes} and further considered in \cite{10.1109/TIT.2017.2654456}. These codes are defined by the property that the cardinality of any canonical projection is a perfect power of the alphabet size. This uniform projection property guarantees that an almost affine code naturally induces a matroid. By dropping the strict requirement of linearity while preserving this underlying combinatorial structure, almost affine codes capture a broader class of geometric incidence structures than classical linear codes. Furthermore, loopless almost affine codes correspond directly to ideal perfect secret sharing schemes via the framework of matroid ports \cite{Lehman,brickell_davenport,farre_padro2007}. In algebraic coding theory, rank-metric codes serve as the $q$-analogue of classical Hamming-metric codes, where the subset lattice of a finite set is replaced by the subspace lattice of a finite-dimensional vector space over a finite field. Extending the combinatorial perspective of matroids to this setting naturally leads to the theory of $q$-matroids \cite{JurriusPellikaan2018,Gorla_2019,Shiromoto}. Just as a linear Hamming-metric code determines a representable matroid, a linear rank-metric code induces a $q$-matroid whose rank function captures essential algebraic invariants of the code. Many fundamental parameters, such as generalized rank weights, duality, and the maximum rank distance property, admit direct combinatorial interpretations through their associated $q$-matroid. Furthermore, much like classical matroids model the access structures of linear secret sharing schemes, $q$-matroids provide a combinatorial framework for analyzing secure network communications and information leakage \cite{qPorts1,qPorts2}.

Recently, Alfarano and Degen introduced a $q$-analogue of multilinear representability for $q$-matroids via $\F_q$-linear matrix rank-metric codes under a divisibility condition \cite{alfarano2026representabilityqmatroidsrankmetriccodes}. Our focus instead is on almost affine vector rank-metric codes beyond the linear setting, from which we develop the associated $q$-matroid theory along with its coding-theoretic and geometric consequences. In particular, our examples yield constructions of multilinear $q$-matroids.

This paper generalizes almost affine codes to the rank metric. We prove that these codes naturally induce $q$-matroids, extending the classical correspondence between almost affine codes and matroids into the $q$-analogue. Building on this, we show that once the $q$-matroid framework is established, many classical structural results translate relatively trivially to the rank-metric setting. Specifically, we adapt the operations of puncturing and shortening, characterize the circuits of the dual $q$-matroid via minimal codewords, and demonstrate that determining the weight and dual distance distributions reduces to straightforward $q$-matroid counting arguments. Applying this framework to information leakage in linear network coding, we demonstrate that almost affine rank-metric codes yield perfect $q$-matroid ports. Notably, we show that a disconnected $q$-matroid can induce a connected $q$-matroid port, which is a divergence from classical matroid theory where connected matroids yield connected ports. Furthermore, we prove that the rank-metric weight distribution and dual distance distribution of almost affine rank-metric codes are determined by the induced $q$-matroid. Lastly, we demonstrate that specific parameter regimes of Additive Generalized Twisted Gabidulin (AGTG) codes yield simple, strictly almost affine rank-metric codes, and we provide a secondary explicit construction derived from proper finite semifields.

The remainder of this paper is organized as follows. Section~\ref{sec:rankmetriccodes} establishes preliminary definitions of rank-metric codes, support spaces, and $q$-matroids. Section~\ref{sec:aa} introduces almost affine rank-metric codes, proves their relationship with $q$-matroids, and develops the operations of puncturing, shortening, and duality. Section~\ref{sec:applicatons} applies this framework to information leakage in network codes and analyses the connectivity of the resulting $q$-matroid ports. Section~\ref{sec:weight} considers the weight distribution and dual distance distribution of almost affine rank-metric codes and shows they are determined by the underlying $q$-matroid. Finally, Section~\ref{sec:geometry} formally defines partial affine $q$-geometries, establishes their one-to-one correspondence with simple almost affine rank-metric codes, and provides explicit algebraic constructions utilizing Additive Generalized Twisted Gabidulin codes and finite proper semifields.

\section{Rank-Metric Codes and \textit{q}-Matroids} \label{sec:rankmetriccodes}
In this section, we establish the fundamental algebraic notation used for rank-metric codes and $q$-matroids. Let $q$ be a prime power and $\mathbb{F}_q$ the finite field of order $q$. We denote the natural numbers, including zero, by $\mathbb{N} = \{0,1,2,\ldots\}$. For finite-dimensional vector spaces $V$ and $E$ over $\mathbb{F}_q$, we write $V \leq E$ to indicate that $V$ is a subspace of $E$. The lattice of all subspaces of $E$, ordered by inclusion, is denoted $\mathcal{L}(E)$. For $V\leq E$, let $V^\perp$ denote the orthogonal complement of $V$ with respect to a fixed non-degenerate symmetric bilinear form $\langle\cdot,\cdot\rangle$ on $E$. For integers $n\geq k\geq 0$, the \emph{Gaussian binomial coefficient}, or $q$-binomial coefficient, counts the number of $k$-dimensional subspaces of an $n$-dimensional vector space over $\F_q$, and is defined as
\[
    \qbinom{n}{k}_q = \begin{cases}
        \prod^{k-1}_{i=0} \frac{q^n-q^i}{q^k-q^i}, & \text{if } 0\leq k \leq n, \\
        0, & \text{otherwise.}
    \end{cases}
\]
Finally, for a subspace $U \le \mathbb{F}_q^n$, the space $U \otimes_{\mathbb{F}_q}\mathbb{F}_{q^m}\leq \F_{q^m}^n$ denotes the standard tensor product used for the extension of scalars to $\mathbb{F}_{q^m}$.

Let $\Pi = \{\gamma_1, \ldots, \gamma_m\}$ be a basis of $\mathbb{F}_{q^m}$ over $\mathbb{F}_q$. For $v \in \mathbb{F}_{q^m}^n$, let $\Pi(v) \in \mathbb{F}_q^{n \times m}$ denote the matrix whose $(i, j)$th entry is the $j$th coordinate of $v_i$ with respect to the basis $\Pi$. The map $v \mapsto \Pi(v)$ defines an $\mathbb{F}_q$-linear isomorphism from $\mathbb{F}_{q^m}^n$ to $\mathbb{F}_q^{n \times m}$. The \emph{support} of a vector $v \in \mathbb{F}_{q^m}^n$, denoted $\supp(v)$, is defined as the column space of the matrix $\Pi(v)$, meaning $\supp(v) = \textup{colsp}(\Pi(v)) \leq \mathbb{F}_q^n$. The rank of $v$ is therefore $\rk(v) = \dim(\supp(v))$. This induces the rank-metric as $d(x,y) = \rk(x-y)$ for all $x,y\in \F_{q^m}^n$.

\begin{definition}
A \emph{rank-metric code} $\C$ is a subset of $\F_{q^m}^n$ endowed with the rank-metric. Its length is $n$, its \emph{dimension} is $\log_{q^m} |\C|$, and its \emph{minimum distance} is $d(\C)\coloneqq \min_{x,y\in \C, x\neq y} d(x,y)$.
\end{definition}

The \emph{Singleton bound} for a rank-metric code $\mathcal{C} \subseteq \mathbb{F}_{q^m}^n$ with $n \leq m$ establishes that $d(\mathcal{C}) \leq n - \log_{q^m}|\mathcal{C}| + 1$. A code achieving equality in this bound is called a \emph{maximum rank distance} (MRD) code.

Unless stated otherwise, the term code will refer to a rank-metric code. References to classical Hamming codes will be made explicitly when relevant.

To formulate a coordinate-free definition of almost affine rank-metric codes, it is necessary to consider the code under canonical projections. Let $V\leq \F_q^n$ be a subspace, and let $G_V$ be a matrix whose rows span $V$. Multiplying the elements of a code $\C\subseteq \F_{q^m}^n$ by $G_V$ acts as an evaluation map. To replace this matrix multiplication with a projection onto a quotient space, we first characterize the kernel of this map.

\begin{lemma} \label{lem:projectkernel}
    Let $V\leq \F_q^n$, and let $G_V$ be a matrix whose rows span $V$. For any vector $v\in \F_{q^m}^n$, the following are equivalent.
    \begin{enumerate}
        \item $G_V v=0$.
        \item $\supp(v) \leq V^\perp$.
        \item $v\in V^\perp \otimes_{\F_q} \F_{q^m}$.
    \end{enumerate}
\end{lemma}
\begin{proof}
We first show the equivalence of 1 and 2. Treating $v$ as a column vector in $\F_{q^m}^n$, the product $G_V v$ evaluates to zero if and only if $G_V$ annihilates every column of the coordinate matrix $\Pi(v)$. This implies that every vector in the column space of $\Pi(v)$, which is by definition $\supp(v)$, is orthogonal to the row space $G_V$. Consequently, $G_Vv=0$ if and only if $\supp(v) \leq V^\perp$.

To show the equivalence between 2 and 3, fix an $\F_q$-basis $\{\alpha_1,\ldots,\alpha_m\}$ of $\F_{q^m}$. Every vector $v\in \F_{q^m}^n$ can be written uniquely as $v= \sum^m_{i=1} \alpha_i v_i$, where $v_i \in \F_q^n$. By definition, $\supp(v)$ is the $\F_q$-span of the vectors $v_1,\ldots,v_m$. Suppose now $\supp(v) \leq V^\perp$. Each vector $v_i$ in the above expansion then satisfies $v\in V^\perp$. Consequently, $v$ is obtained as an $\F_{q^m}$-linear combination of vectors from $V^\perp$. The set of all such combinations is precisely $V^\perp \otimes_{\F_q} \F_{q^m}$, and hence $v\in V^\perp \otimes_{\F_q} \F_{q^m}$.

Conversely, suppose that $v\in V^\perp \otimes_{\F_q} \F_{q^m}$. Then $v$ can be written in the form $v=\sum_j \beta_j w_j$, where $\beta_j\in \F_{q^m}$ and $w_j\in V^\perp$. Applying the matrix $G_V$ gives $G_Vv = \sum_j \beta_j(G_Vw_j)$. Since each $w_j$ is orthogonal to the row space of $G_V$, we have $G_Vv=0$, which as established is equivalent to $\supp(v)\leq V^\perp$.
\end{proof}

Linear rank-metric codes naturally induce $q$-matroids. We omit the specific details of the linear case here, as Section~\ref{sec:aa} generalizes this to almost affine codes. For a comprehensive treatment of $q$-matroids and their derivation from linear rank-metric codes, we refer the reader to \cite{DBLP:journals/corr/abs-2104-06570, BYRNE2022149, BYRNE2024105799, vertigan}. {For the remainder of this section let $E$ denote an $n$-dimensional vector space over $\F_q$.}

\begin{definition} A \textit{$q$-matroid} is a pair $\mathcal{M} = (E,\rho)$ where $\rho\colon \mathcal{L} (E) \rightarrow \mathbb{N}$ is a \textit{rank function} such that the following axioms hold.
\begin{enumerate}
    \item[(R1)] \textit{Boundedness}: $0\leq \rho(V) \leq \dim V$ for all $V\leq E$.
    \item[(R2)] \textit{Monotonicity}: $\rho(V)\leq \rho(W)$ for all $V,W\leq E$ with $V\leq W$.
    \item[(R3)] \textit{Submodularity}: $\rho(V + W) + \rho(V\cap W ) \leq \rho(V) + \rho(W)$ for all $V,W\leq E$.
\end{enumerate}
The \emph{rank} of $\M$ is $\rho(E)$. \label{def:qmatroid}
\end{definition}

An equivalent way to define the rank function $\rho\colon \L(E)\rightarrow \mathbb{N}$ of a $q$-matroid is via the local rank axioms (R1')-(R3') given below \cite{BYRNE2022149}.

\begin{enumerate}
    \item[(R1')] $\rho(\{0\})=0$.
    \item[(R2')] $\rho(V) \leq \rho(V+x) \leq \rho(V)+1$ for all $V\leq E$ and $x\leq E$ with $\dim x =1$.
    \item[(R3')] If $\rho(V) = \rho(V+x) = \rho(V+y)$, then $\rho(V+x+y)=\rho(V)$ for all $V,x,y\leq E$ with $\dim x = \dim y=1$.
\end{enumerate}

{
\begin{example}
    Let $0\leq k \leq n$ be an integer. Define $\rho\colon \L(E) \rightarrow \NN$ by $V\mapsto \min\{ \dim V, k\}$. Then $\mathcal{U}_{n,k} = (E,\rho)$ is a $q$-matroid, called the \emph{uniform $q$-matroid} of rank $k$.
\end{example}}

Let $\mathcal{M}=(E,\rho)$ be a $q$-matroid, and let $V \leq E$. We say that $V$ is \emph{independent} if $\rho(V)=\dim V$. We call $V$ a \emph{basis} if $V$ is independent and $\rho(V)=\rho(E)$. We call $V$ a \emph{circuit} if $V$ is dependent and every proper subspace of $V$ is independent. A $1$-dimensional circuit is called a \emph{loop}, and $\mathcal{M}$ is said to be \emph{simple} if it has no $1$- or $2$-dimensional circuits.

{ The following are some standard operations on $q$-matroids.

\begin{definition}
    Let $\mathcal{M} = (E,\rho)$ be a $q$-matroid. Define $\rho^*(V) \coloneqq \dim V - \rho(E) + \rho(V^{\perp})$ for all $V\leq E$. Then $\mathcal{M}^* \coloneqq (E,\rho^*)$ is a $q$-matroid called the \textit{dual} of $\mathcal{M}$, and $(\mathcal{M}^*)^* = \mathcal{M}$.
\end{definition}}

\begin{definition} \label{def:equivalentmatroids}
Let $\mathcal{M}_i = (E_i, \rho_i)$ be $q$-matroids for $i=1,2$. Then $\mathcal{M}_1$ and $\mathcal{M}_2$ are \textit{equivalent}, denoted by $\mathcal{M}_1 \simeq \mathcal{M}_2$, if there exists an $\mathbb{F}_q$-linear isomorphism $\phi \colon E_1 \rightarrow E_2$ such that $\rho_2(\phi(V)) = \rho_1(V)$ for all $V\leq E_1$. 
\end{definition}

We shall also consider contractions and restrictions of $q$-matroids, and we will later see how these relate to shortenings and puncturings of almost affine codes.

\begin{definition} Let $\mathcal{M} = (E,\rho)$ be a $q$-matroid and let $Z\leq E$. Then $\mathcal{M}|_Z \coloneqq (Z,\rho|_{Z})$ where $\rho|_{Z}(V) \coloneqq \rho(V)$ for all $V\leq Z$ is a $q$-matroid, called the \textit{restriction of $\mathcal{M}$ to $Z$}. Additionally, let $\sigma \colon E\rightarrow E/Z$ denote the canonical projection. Then $\mathcal{M}/Z \coloneqq (E/Z,\rho_{E/Z})$ where $\rho_{E/Z}(V) \coloneqq \rho(\sigma ^{-1}(V))-\rho(Z)$ is a $q$-matroid, called the \textit{contraction of $\mathcal{M}$ by $Z$}.
\end{definition}

\section{Almost Affine Rank-Metric Codes} \label{sec:aa}

While linear rank-metric codes naturally induce $q$-matroids through their vector space structure, extending this relationship to non-linear codes requires a property that guarantees uniform projection sizes. This motivates the coordinate-free definition of almost affine rank-metric codes.

An almost affine rank-metric code $\C\subseteq \F_{q^m}^n$ is a rank-metric code such that every projection has a size which is a power of $q^m$. As established in Lemma~\ref{lem:projectkernel}, two elements in $\C$ share the same image under the evaluation map $G_V$ if and only if their difference lies in the kernel $V^\perp \otimes_{\F_q} \F_{q^m}$. Consequently, they belong to the same coset of $V^\perp \otimes_{\F_q} \F_{q^m}$. The size of the image $|G_V\C|$ is therefore the cardinality of the projection of $\C$ onto the quotient space $\F_{q^m}^n /(V^\perp \otimes_{\F_q} \F_{q^m})$. This equivalence allows the almost affine property to be defined purely in terms of these quotient spaces, removing the dependency of a specific basis of $V$.

\begin{definition}
Let $\mathcal{C} \subseteq \mathbb{F}_{q^m}^n$ be a rank-metric code. For any subspace $V \le \mathbb{F}_q^n$, let $V^\perp$ denote its orthogonal complement in $\mathbb{F}_q^n$. Let $\pi_V \colon \mathbb{F}_{q^m}^n \to \mathbb{F}_{q^m}^n / (V^\perp \otimes_{\F_q}\mathbb{F}_{q^m})$ be the canonical projection mapping $c \mapsto c + (V^\perp \otimes_{\F_q}\mathbb{F}_{q^m})$. The code $\mathcal{C}$ is called an \emph{almost affine rank-metric code} if for all $V \le \mathbb{F}_q^n$,
$$ \log_{q^m} |\pi_V(\mathcal{C})| \in \mathbb{N}. $$
\end{definition}

All $\F_{q^m}$-linear or affine vector rank-metric codes are almost affine. Additionally, every almost affine rank-metric code is also almost affine in the Hamming metric sense.

\begin{theorem}
Let $\C\subseteq \F_{q^m}^n$ be almost affine. For each $V\leq \F_q^n$ define $$\rho_\C(V) = \log_{q^m} |\pi_V(\mathcal{C})|.$$ Then $\mathcal{M}_\C = (\F_q^n,\rho_\C)$ is a $q$-matroid.
\end{theorem}
\begin{proof} We shall prove it using the local rank axioms. For (R1') let $V=\{0\}$, so $V^\perp = \F_q^n$, so $V^\perp \otimes_{\F_q}\F_{q^m} = \F_{q^m}^n$. Thus, $\pi_{\{0\}}(\C)$ consists of a single coset, so $\rho_\C(\{0\})=0$.

For (R2') consider $x\leq \F_q^n$ with $\dim x =1$ and $\dim (V+x) \neq \dim V$. As $V\leq V+x$ implies $(V+x)^\perp \leq V^\perp$, so $(V+x)^\perp \otimes_{\F_q}\F_{q^m} \leq V^\perp \otimes_{\F_q}\F_{q^m}$. This induces a natural surjective homomorphism $\phi \colon \F_{q^m}^n/((V+x)^\perp \otimes_{\F_q}\F_{q^m}) \rightarrow \F_{q^m}^n/(V^\perp \otimes_{\F_q}\F_{q^m})$ defined by $\phi(s+(V+x)^\perp \otimes_{\F_q}\F_{q^m}) = s + V^\perp \otimes_{\F_q}\F_{q^m}$. Restricting this map to the projections of $\C$ yields a surjection from $\pi_{V+x}(\C)$ onto $\pi_V(\C)$, so $|\pi_V(\C)|\leq |\pi_{V+x}(\C)|$, so $\rho_\C(V) \leq \rho_\C(V+x)$. Lastly, adding a one-dimensional subspace $x$ to $V$ reduces the dimension of the orthogonal complement by one over $\F_q$, so $\dim_{\F_{q^m}} \left((V)^\perp \otimes_{\F_q}\F_{q^m}\right) / \left( (V+x)^\perp \otimes_{\F_q}\F_{q^m}\right) = 1$. As such, each coset in $\pi_V( \C)$ partitions into at most $q^m$ distinct cosets in $\pi_{V+x}(\C)$, which yields the bound $|\pi_{V+x}(\C)|\leq q^m |\pi_V(\C)|$ which finalises the argument.

For (R3’) let $V \le \F^n_q$ and $x,y \le \F^n_q$ such that $\dim x = \dim y = 1$ and $\rho_\C (V)=\rho_\C(V+x)=\rho_\C(V+y).$ We therefore have $|\pi_V(\C)|=|\pi_{V+x}(\C)|=|\pi_{V+y}(\C)|$.
Since the natural maps from $\pi_{V+x}(\C)$ and $\pi_{V+y}(\C)$ onto $\pi_V(\C)$ are surjective, they are both bijective. Hence, for any $c \in \C$, the cosets $c+(V+x)^\perp \otimes_{\F_q}\F_{q^m}$ and $c+(V+y)^\perp \otimes_{\F_q}\F_{q^m}$ are uniquely determined by the coset $c+V^\perp \otimes_{\F_q}\F_{q^m}$. As $(V+x+y)^\perp=(V+x)^\perp \cap (V+y)^\perp$ we get
\[
(V+x+y)^\perp \otimes_{\F_q}\F_{q^m} = \big((V+x)^\perp \otimes_{\F_q}\F_{q^m}\big)\cap \big((V+y)^\perp \otimes_{\F_q}\F_{q^m}\big).
\]
Thus the coset $c+(V+x+y)^\perp \otimes_{\F_q}\F_{q^m}$ is the intersection of the two cosets above, and is therefore uniquely determined by $c+V^\perp \otimes_{\F_q}\F_{q^m}$. This shows that the natural surjection from $\pi_{V+x+y}(\C)$ to $\pi_V(\C)$ is injective. Hence it is bijective, so $|\pi_{V+x+y}(\C)|=|\pi_V(\C)|$ which completes the proof.
\end{proof}

{Because an MRD code achieves the Singleton bound, any canonical projection onto a subspace of dimension $k$ or less preserves the cardinality of the code. Consequently, the induced $q$-matroid of any $k$-dimensional almost affine MRD code $\mathcal{C}\subseteq \mathbb{F}_{q^m}^n$ satisfies $\mathcal{M}_{\mathcal{C}} = \mathcal{U}_{n,k}$.}

\begin{example} 
Consider the 8-dimensional vector space $\mathcal{D} \leq \F_2^{3 \times 4}$ with basis
\begin{align*}
\Big\{&\begin{bsmallmatrix}
1 & 0 & 0 & 0 \\
0 & 0 & 0 & 0 \\
0 & 1 & 1 & 0
\end{bsmallmatrix}, 
\begin{bsmallmatrix}
0 & 1 & 0 & 0 \\
0 & 0 & 0 & 0 \\
0 & 1 & 1 & 1
\end{bsmallmatrix}, 
\begin{bsmallmatrix}
0 & 0 & 1 & 0 \\
0 & 0 & 0 & 0 \\
1 & 0 & 0 & 1
\end{bsmallmatrix}, 
\begin{bsmallmatrix}
0 & 0 & 0 & 1 \\
0 & 0 & 0 & 0 \\
1 & 1 & 0 & 1
\end{bsmallmatrix},
\begin{bsmallmatrix}
0 & 0 & 0 & 0 \\
1 & 0 & 0 & 0 \\
1 & 0 & 0 & 0 
\end{bsmallmatrix}, 
\begin{bsmallmatrix}
0 & 0 & 0 & 0 \\
0 & 1 & 0 & 0 \\
0 & 1 & 0 & 0
\end{bsmallmatrix}, 
\begin{bsmallmatrix}
0 & 0 & 0 & 0 \\
0 & 0 & 1 & 0 \\
0 & 0 & 1 & 0
\end{bsmallmatrix}, 
\begin{bsmallmatrix}
0 & 0 & 0 & 0 \\
0 & 0 & 0 & 1 \\
0 & 0 & 0 & 1
\end{bsmallmatrix}\Big\}.
\end{align*}
Fix an $\F_2$-basis $\Pi = \{\gamma_1,\dots,\gamma_4\}$ of $\F_{2^4}$ and let $\Pi^{-1}\colon \F_2^{3\times 4}\rightarrow \F_{2^4}^3$ denote the induced coordinate map. Define $\C = \Pi^{-1}(\mathcal{D})\subseteq \F_{2^4}^3$. Direct computation (e.g., using Magma) verifies that $\C$ is a $2$-dimensional almost affine rank-metric code. Because $\C$ contains the zero codeword yet is not an $\F_{2^4}$-linear space, it is a strictly almost affine code. However, by \cite[Theorem 6.10]{alfarano2026representabilityqmatroidsrankmetriccodes}, there exists a $\F_{2^4}$-linear code $\C'\leq \F_{2^4}^3$ such that $\M_\C = \M_{\C'}$.
\label{exmp:almostaffine}
\end{example}

We say that a $q$-matroid $\M$ is \emph{almost affinely representable} if there exists an almost affine code $\C$ such that $\M\simeq \M_\C$.

The Vámos matroid is a canonical example in matroid theory, notable for its pathological representability properties. In particular, it is neither representable over any field nor almost affinely representable. The $q$-analogue of the Vámos matroid was introduced in \cite{DBLP:journals/corr/abs-2104-06570}, where it was shown there exists no $\F_{q^m}$-linear rank-metric code $\C\subseteq \F_{q^m}^n$ whose induced $q$-matroid is the Vámos $q$-matroid.

\begin{example}[Vámos $q$-matroid] Let $$
\mathcal{Y} = \big\{ 
\{1,2,3,4\},\;
\{1,4,5,6\},\;
\{2,3,5,6\},\;
\{1,4,7,8\},\;
\{2,3,7,8\}
\big\}.
$$

Fix a basis $\mathcal{B} = \{b_1,\ldots,b_8\}$ of $\F_q^8$ and let $\xi \colon \{1,\ldots,8\} \rightarrow \mathcal{B}$ be a bijection. Let $$\mathcal{S} = \big\{ \langle \xi(Y) \rangle_{\mathbb{F}_q} \,\big|\, Y \in \mathcal{Y} \big\},$$
and define the function $\rho \colon \mathcal{L}(\mathbb{F}_q^8) \to \mathbb{Z}_{\geq 0}$ given by
$$
\rho(V) =
\begin{cases}
3, & \text{if } V \in \mathcal{S},\\[4pt]
\min\{\dim V,\,4\}, & \text{otherwise.}
\end{cases}
$$
Then $\M = (\F_q^8,\rho)$ is a $q$-matroid, called the Vámos $q$-matroid.
\end{example}

\begin{corollary}
The Vámos $q$-matroid is not almost affinely representable.
\end{corollary}
\begin{proof}
Almost affine representability of the Vámos $q$-matroid would imply almost affine representability of the Vámos matroid, which would be a contradiction to \cite[Proposition 9]{AlmostAffineCodes}.
\end{proof}

\begin{definition}
Let $\C_1,\C_2\subseteq \F_{q^m}^n$ be almost affine. $\C_1$ and $\C_2$ are equivalent, denoted $\C_1\simeq \C_2$, if there exists $A\in \GL_n(\F_q)$ and $t\in \F_{q^m}^n$ such that $\C_2 = A\C_1 + t$.
\end{definition}

Equivalent almost affine codes yield equivalent $q$-matroids.

\begin{proposition}\label{prop:equivalence}
Let $\C_1,\C_2\subseteq \F_{q^m}^n$ be almost affine. If $\C_1\simeq \C_2$, then $\M_{\C_1} \simeq \M_{\C_2}$.
\end{proposition}
\begin{proof}
By definition, there exists $A\in \GL_n(\F_q)$ and $t\in \F_{q^m}^n$ such that $\C_2 = A\C_1+t$. Define the map $\phi \colon \F_q^n\rightarrow \F_q^n$ by $\phi(V)= (A^{-1})^\top V$. This map is clearly an $\F_q$-linear isomorphism, and we must show that $\rho_2(\phi(V)) = \rho_1(V)$ for all $V\leq \F_q^n$. Recall that $\rho_i(W) = \log_{q^m} |\pi_W(\C_i)|$, where $|\pi_w(\C_i)|$ is the number of distinct cosets of $W^\perp \otimes_{\F_q}\F_{q^m}$ intersecting $\C_i$. Note that $\phi(V)^\perp = ((A^{-1})^\top V)^\perp = A(V^\perp)$, so $\phi(V)^\perp \otimes_{\F_q}\F_{q^m} = A(V^\perp \otimes_{\F_q}\F_{q^m})$.

Two elements $x_2,y_2\in \C_2$ can be uniquely written as $x_2 = Ax_1+t$ and $y_2 = Ay_1 + t$, with $x_1,y_1\in \C_1$, so $x_2-y_2 = A(x_1-y_1)$. As we have $A(x_1-y_1) \in A(V^\perp \otimes_{\F_q}\F_{q^m})$ if and only if $x_1-y_1 \in V^\perp \otimes_{\F_q}\F_{q^m}$, the affine transformation $x\mapsto Ax+t$ induces a bijection between the cosets in $\pi_V(\C_1)$ and the cosets in $\pi_{\phi(V)}(\C_2)$, which proves the result.
\end{proof}

As established, evaluating a code by multiplying it with a matrix $G_V$ yields a set of vectors $G_V\C$ that is in bijection with the set of cosets of $\pi_V(\C)$ in the quotient space $\F_{q^m}^n/(V^\perp \otimes_{\F_q}\F_{q^m})$. This allows us to define standard rank-metric operations, such as puncturing and shortening through canonical projections and quotient spaces.

\begin{definition}
    Let $\C\subseteq \F_{q^m}^n$ be almost affine, and let $Z\leq \F_q^n$. Then $\C_Z\coloneqq \pi_Z(\C)$ is the \emph{puncturing} of $\C$ on $Z$.
\end{definition}

\begin{proposition}
Let $\C\subseteq \F_{q^m}^n$ be almost affine, and let $Z\leq \F_q^n$. Then $(\M_\C)|_Z = \M_{\C_Z}$.
\end{proposition}
\begin{proof}
    By definition, $\rho_{\C_Z}$ is the rank function of a $q$-matroid with ground space $Z$. Now, consider $V\leq Z$, so we wish to determine $\rho_{\C_Z}(V)$. As $Z^\perp \leq V^\perp$ we have the containment $Z^\perp \otimes_{\F_q}\F_{q^m} \leq V^\perp \otimes_{\F_q}\F_{q^m}$, which induces a natural surjective homomorphism $\tilde{\pi}_V\colon \F_{q^m}^n/(Z^\perp \otimes_{\F_q}\F_{q^m})\rightarrow \F_{q^m}^n/(V^\perp \otimes_{\F_q}\F_{q^m})$ defined by $x + (Z^\perp \otimes_{\F_q}\F_{q^m}) \mapsto x + (V^\perp \otimes_{\F_q}\F_{q^m})$. By definition, $\C_Z = \pi_Z(\C)$ so $\tilde{\pi}_V \circ \pi_Z = \pi_V$, therefore $\tilde{\pi}_V(\C_Z) = \tilde{\pi}_V(\pi_Z(\C)) = \pi_V(\C)$, which proves the result. 
\end{proof}

For linear rank-metric codes, shortening with respect to a subspace $Z\leq \F_q^n$ produces the subcode consisting of codewords whose support is contained within $Z^\perp$. This operation restricts the code to elements that evaluate to zero under the projection $\pi_Z$. Since almost affine codes generalize affine spaces they do not necessarily contain the zero codeword. Consequently, defining a shortened code by strictly requiring $\pi_Z(c)=0$ might yield an empty set. To ensure a well-defined subcode, we instead anchor the operation to a fixed codeword $x\in \C$. We define the shortened code by collecting all elements $c\in \C$ for which $\supp(c-x)\leq Z^\perp$. This condition translates to the equality $\pi_Z(c) = \pi_Z(x)$, meaning the generalized shortened code forms an affine translation of a support space. To define shortened subcodes we first consider the following subcodes. For an almost affine code $\C\subseteq \F_{q^m}^n$, $x\in \C$, and $Z\leq \F_q^n$ we define such a subcode as
\[
    \C(Z,x) \coloneqq \{ c\in \C\mid \pi_Z(c)=\pi_Z(x)\}.
\]
The size of the subcode $\C(Z,x)$ is independent of the choice of $x$, as the following result shows.
\begin{proposition}\label{prop:shortenedsize}
Let $\C\subseteq \F_{q^m}^n$ be almost affine, $Z\leq \F_q^n$, and $x\in \C$. Then $$|\C(Z,x)|=q^{m(\dim \C - \rho_\C(Z))}.$$ 
\end{proposition}
\begin{proof}
Denote $k= \dim \C$ and $r = \rho_\C(Z)$. Let $U\leq Z$ be an independent space in $\M_\C$ such that $\rho(U) = \rho(Z)=r$. Then
\[
    \C(Z,x) = \{c\in \C \mid \pi_U(c) =\pi_U(x)\}.
\]
Now extend $U$ to a basis $B$ of $\M_\C$. Then $\rho_\C(B) = \dim B = k$, so $|\pi_B(\C)| = q^{mk}$. As $\F_{q^m}^n/(B^\perp \otimes_{\F_q}\F_{q^m})$ also has cardinality $q^{mk}$, it follows that $\pi_B(\C)$ equals this quotient space. Hence the map $\pi_B \colon \C \rightarrow \F_{q^m}^n/(B^\perp \otimes_{\F_q}\F_{q^m})$ is bijective. As $U\leq B$, we have $B^\perp \otimes_{\F_q}\F_{q^m}\leq U^\perp \otimes_{\F_q}\F_{q^m}$, and there exists a canonical surjective linear map $\phi\colon \F_{q^m}^n/(B^\perp \otimes_{\F_q}\F_{q^m})\rightarrow \F_{q^m}^n/(U^\perp\otimes_{\F_q}\F_{q^m})$ satisfying $\pi_U = \phi \circ \pi_B$ on $\C$. Its kernel is $(U^\perp \otimes_{\F_q}\F_{q^m})/(B^\perp\otimes_{\F_q}\F_{q^m})$ which has $\F_{q^m}$-dimension $k-r$. Hence every fibre of $\phi$ has cardinality $q^{m(k-r)}$. Since $\pi_B$ is bijective, the same fibre size holds for $\pi_U$ on $\C$.
\end{proof}

\begin{corollary}
Let $\C\subseteq \F_{q^m}^n$ be almost affine. Let $Z\leq \F_q^n$ and $x\in \C$. Then $\C(Z,x)$ is an almost affine subcode of $\C$. \label{cor:shortened}
\end{corollary}
\begin{proof}
Let $W\leq \F_q^n$. For any $c\in \C(Z,x)$, the projection $\pi_Z(c)$ is fixed to $\pi_Z(x)$. As $\pi_{Z+W}(c)$ is uniquely determined by $\pi_Z(c)$ and $\pi_W(c)$ we obtain
\begin{align*}
    |\C(Z,x)_W| = |\C(Z,x)_{Z+W}| = |\C_{Z+W}(Z,\pi_{Z+W}(x))| = q^{m(\rho_\C(Z+W) - \rho_\C(Z))},
\end{align*}
from which the result follows.
\end{proof}

\begin{definition}\label{defn:shorten}
    Let $\C\subseteq \F_{q^m}^n$ be almost affine, and $x\in \C$. Let $Z,C\leq \F_q^n$ such that $Z\oplus C = \F_q^n$. Then $\C^{Z,C}(x)\coloneqq \C(Z,x)_{C}$ is the \emph{shortening} of $\C$ by $Z$.
\end{definition}

For an almost affine code $\C\subseteq \F_{q^m}^n$ and a fixed $Z\leq \F_q^n$ the shortening by $Z$ does not depend on the choice of complement, so we omit the complement from the superscript.
The elements of the shortened code $\C^{Z,C}(x)$ depend on the chosen complement $C$. However, for any two direct complements $C_1$ and $C_2$, the canonical projection from $C_1$ to $C_2$ parallel to $Z$ establishes an affine isomorphism between $\C^{Z,C_1}(x)$ and $\C^{Z,C_2}(x)$. We summarise this in Proposition \ref{prop:shorteninvariant}. %

\begin{proposition} \label{prop:shorteninvariant}
Let $\C\subseteq \F_{q^m}^n$ be almost affine, and $x\in \C$. Let $Z,C_1,C_2\leq \F_q^n$ such that $Z\oplus C_1 = Z\oplus C_2 = \F_q^n$. Then $\C^{Z,C_1}(x)\simeq \C^{Z,C_2}(x)$.
\end{proposition}

By Proposition~\ref{prop:equivalence}, equivalent codes induce equivalent $q$-matroids, so the resulting $q$-matroid structure is invariant under the choice of complement, and we shall omit the complement from the notation and simply write $\M_{\C^Z(x)}$. Furthermore, the size of $\C(Z,x)_W$ for any $W\leq Z^\perp$ is independent of the choice of $x\in \C$ by Proposition~\ref{prop:shortenedsize}. As such, for any $x,y\in \C$ then $\M_{\C^Z(x)} = \M_{\C^Z(y)}$, so we write $\M_{\C^Z}$ to mean any of these $q$-matroids.

To define the support of an almost affine rank-metric code we must again define it with respect to some fixed codeword $x\in \C$ as we do not have a canonical choice. Thus, we denote the $x$-support of an element $c\in \C$ as $$\supp_{x}(c)\coloneqq \supp(c-x).$$

\begin{definition}
Let $\C\subseteq \F_{q^m}^n$ be almost affine, and fix $x\in \C$. The $x$-support of $\C$ is $\supp_x(\C) = \sum _{c\in \C} \supp_x(c)$.
\end{definition}

\begin{lemma}
    Let $\C\subseteq \F_{q^m}^n$ be almost affine, and fix $x,y\in \C$. Then $\supp_x(\C) = \supp_y(\C)$.
\end{lemma}
\begin{proof}
    For any $c \in \mathcal{C}$, we can write $c - y = (c - x) - (y - x)$. The support is subadditive, so
    \[\supp_y(c) = \supp(c - y) \leq  \supp(c - x) + \supp(y - x).\]
    Because both $c$ and $y$ belong to $\mathcal{C}$, the individual supports $\supp_x(c) = \supp(c - x)$ and $\supp_x(y) = \supp(y - x)$ are both contained within $\supp_x(\mathcal{C})$ by definition. Consequently, $\supp_y(c) \subseteq \supp_x(\mathcal{C})$ for every codeword $c \in \mathcal{C}$. Summing these subspaces over all $c \in \mathcal{C}$ establishes the containment $\supp_y(\mathcal{C}) \subseteq \supp_x(\mathcal{C})$. By symmetry, reversing the roles of $x$ and $y$ yields $\supp_x(\mathcal{C}) \subseteq \supp_y(\mathcal{C})$, completing the proof.
\end{proof}
Thus, the support of any almost affine rank-metric code is well-defined as long as the support is taken with respect to a codeword of the given code. We therefore write $\supp \C = \supp_x(\C)$ for any $x\in \C$.

As with linear rank-metric codes we can describe the circuits of the dual $q$-matroid via the notion of minimal codewords. Minimality in this case is described as minimal support with respect to the $0$ codeword, but as per previous discussions we must here define minimality with respect to a fixed codeword $x\in \C$. 

\begin{definition}
Let $\C\subseteq \F_{q^m}^n$ be almost affine, and fix $x\in \C$. A codeword $c\in \C\setminus\{x\}$ is $x$-\emph{minimal} if there exists no codeword $c'\in \C\setminus \{x\}$ such that $\supp_x(c')<\supp_x(c)$.
\end{definition}

For the proof of Proposition~\ref{prop:dualcircuits} we will use the following description of independent sets of the dual $q$-matroid of an almost affine rank-metric code $\C$. 
\begin{align*}
    Y \text{ is independent in } \M_\C^* &\Leftrightarrow \rho_\C^*(Y) = \dim Y = \dim Y - \dim \C + \rho_\C(Y^\perp) \\
    &\Leftrightarrow \rho_\C(Y^\perp) = \dim \C \\
    &\Leftrightarrow Y^\perp \text{ contains a basis of } \M_\C \\ 
    &\Leftrightarrow \forall x\in \C \text{ then } |\C(Y^\perp,x)|=1 \\
    &\Leftrightarrow \forall x\in \C \;\nexists c\in \C\setminus\{x\} \colon \supp_x(c)\leq Y.
\end{align*}

\begin{proposition} \label{prop:dualcircuits}
Let $\C\subseteq \F_{q^m}^n$ be almost affine, and fix $x\in \C$. $C\leq \F_q^n$ is a circuit of $\M_{\C}^*$ if and only if $C = \supp_x(c)$ for some $x$-minimal codeword $c\in \C\setminus\{x\}$.
\end{proposition}
\begin{proof}
Suppose $C\leq \F_q^n$ is a circuit of $\M_{\C}^*$ and let $x\in \C$. By our preceding arguments there exist $c\in \C\setminus \{x\}$ such that $\supp_x(c) \leq C$. If $\supp_x(c)<C$, then $Y=\supp_x(c)$ would be a proper dependent subspace of $C$, which would be a contradiction. As such, $C = \supp_x(c)$. Furthermore, if there exists $c'\in \C\setminus \{x\}$ such that $\supp_x(c')<\supp_x(c)$ then $Y = \supp_x(c')$ would again be a dependent proper subspace of $C$. Thus, $c$ is $x$-minimal.

Conversely, let $x\in \C$ and an $x$-minimal codeword $c\in \C\setminus\{x\}$. Let $C=\supp_x(c)$, so $C$ is dependent in $\M_\C^*$. Let $Y<C$ be any proper subspace and suppose for contradiction that $Y$ is dependent. Then there exists $c'\in \C\setminus\{x\}$ such that $\supp_x(c')\leq Y < C$, which contradicts $c$ being $x$-minimal.
\end{proof}

\section{Applications to \emph{q}-Matroid Ports} \label{sec:applicatons}

Having established the local properties and dual circuits of almost affine rank-metric codes, we now explore an immediate cryptographic application of their induced $q$-matroids. Simonis and Ashikhmin \cite{AlmostAffineCodes} established that loopless almost affine Hamming codes give rise to ideal perfect secret sharing schemes. This correspondence translates naturally to the rank-metric setting via the framework of $q$-matroid ports introduced in \cite{qPorts1,qPorts2}.

\begin{definition}
    Let $\M = (E,\rho)$ be a $q$-matroid and $P_0\leq E$ such that $P_0\oplus P =E$ and $\rho(P_0)>0$. We define $\mathbf{S}_{P_0,P}(\M) = (\Gamma,\mathcal{A})$, where
    \begin{align*}
        \Gamma &\coloneqq \{ V\leq P \mid \rho(P_0 + V) = \rho(V) \}, \\
        \mathcal{A} &\coloneqq \{W\leq P \mid \rho(P_0+W) = \rho(P_0) + \rho(W) \}.
    \end{align*}
    $\mathbf{S}_{P_0,P}(\M)$ is called \emph{generalized $q$-matroid port}. If $\dim P_0 = 1$, we call it a \emph{$q$-matroid port}.
\end{definition}

For a generalized $q$-matroid port $\mathbf{S}_{P_0,P}(\M)$ we say it is \emph{perfect} if $\Gamma \cup \mathcal{A} = \L(P)$. As $q$-matroids have integral rank functions then $q$-matroid ports are always perfect. We denote $\Gamma_{\min} = \{ V\in \Gamma \mid \text{for all }W<V\Rightarrow W\notin \Gamma\}$ to be the minimal spaces of $\Gamma$. $\mathbf{S}_{P_0,P}(\M)$ is \emph{connected} if for all $p\leq P$ with $\dim p = 1$ then there exists $V\in \Gamma_{\min}$ such that $p\leq V$. It is \emph{ideal} if for all $p\leq P$ with $\dim p = 1$ then $\rho(p) =1$.

When the $q$-matroid $\M_\C$ is induced by an almost affine code we therefore obtain a perfect $q$-matroid port, and if $\C$ is also assumed to be loopless then we obtain a perfect and ideal $q$-matroid port. However, in the classical matroid case we have that a matroid port is connected if and only if the underlying matroid is connected. This is however not the case for $q$-matroids. To be more precise we first consider connectivity of $q$-matroids as introduced in \cite{conca2026intersectingcodesconnectivityqmatroids}.

\begin{definition}
    Let $\M = (E,\rho)$ be a $q$-matroid and $t>0$ be an integer. A pair $(A,B)$ of subspaces $A,B\leq E$ is a \emph{vertical $t$-separation} of $\M$ if $A\oplus B = E$, $\min\{\rho(A),\rho(B)\}\geq t$, and $\rho(A) + \rho(B) - \rho(E) <t$. A $q$-matroid is \emph{$j$-connected} if it has no $\ell$-separations for $\ell=1,\ldots,j-1$. A $q$-matroid is \emph{disconnected} if it has a $1$-separation.
\end{definition}

In classical matroid theory, a matroid is defined as connected (or strictly, 2-connected) if it lacks 1-separations. An equivalence establishes that a classical matroid is connected if and only if every pair of distinct elements is contained in a common circuit \cite{Oxley}. This combinatorial property is what enforces the classical theorem that a connected matroid yields a connected port \cite{Lehman}: because every participant is guaranteed to share a circuit with the secret, every participant has a structural path to a minimal reconstructing set. In the rank-metric setting, this equivalence collapses. The following example showcases a connected $q$-matroid port which arises from a disconnected $q$-matroid.

\begin{example}
Let $E= \F_2^4$ and choose $\alpha \in \F_{2^2}\setminus \F_2$. Define the rank-metric code $\C\leq \F_{2^2}^4$ as the space spanned by the rows of the matrix
\[
    G = \begin{bmatrix}
        1 & \alpha & 0 & 0 \\
        0 & 0 & 1 & \alpha
    \end{bmatrix}.
\]
$\M_\C$ is clearly disconnected as $A = \langle \mathbf{e}_1,\mathbf{e}_2\rangle_{\F_2}$ and $B = \langle \mathbf{e}_3,\mathbf{e}_4\rangle_{\F_2}$ form a vertical $1$-separation. Consider the $q$-matroid port $\mathbf{S}_{P_0,P}(\M_\C)$ with $P_0 = \langle \mathbf{e}_1 +\mathbf{e}_3\rangle_{\F_2}$ and $P = \langle \mathbf{e}_2,\mathbf{e}_3,\mathbf{e}_4\rangle_{\F_2}$, which is ideal and perfect. Furthermore, as
\begin{align*}
    \Gamma_{\min} = \{\langle \mathbf{e}_2+\mathbf{e}_4\rangle_{\F_2},\; \langle \mathbf{e}_2,\mathbf{e}_3\rangle_{\F_2},\; \langle \mathbf{e}_2+\mathbf{e}_3,\mathbf{e}_4\rangle_{\F_2},\; \langle \mathbf{e}_2,\mathbf{e}_3+\mathbf{e}_4\rangle_{\F_2} \},
\end{align*}
then $\mathbf{S}_{P_0,P}$ is also connected, while it arises from a disconnected $q$-matroid.
\end{example}

\section{Weight Distribution and Dual Distribution} \label{sec:weight}
Returning from cryptographic applications to the structural invariants of the codes themselves, we now address the weight distribution. In classical coding theory, this distribution is an essential combinatorial invariant determined algebraically from the codewords. For almost affine codes, which generally lack a linear structure, these parameters must instead be derived from the induced $q$-matroid. We first establish that the rank-metric weight distribution of an almost affine code is completely determined by the rank function of $\mathcal{M}_{\mathcal{C}}$.

\begin{theorem} \label{thm:x-distance}
    Let $\C\subseteq \F_{q^m}^n$ be almost affine with $k=\dim \C$, and fix a codeword $x\in \C$. For $j,u,v=0,1,\ldots,n$ denote
    \begin{align*}
    A_j(x) &\coloneqq |\{c \in \C \mid \dim \supp_x(c) = j\}|,\\
        R_v^u &\coloneqq |\{ V\leq \F_q^n \mid \rho_\C(V)=u,\; \dim V = v\}|.
    \end{align*}
    Then
    \[
    A_j(x) = \sum_{i = 0}^n(-1)^{i-j}q^{\binom{i-j}{2}}\qbinom{n-j}{n-i}_q \sum^{n}_{u=0} q^{m(k-u)}R_{n-i}^u,
    \]
    and $A_j(x)$ is entirely determined by $\M_\C$.
\end{theorem}
\begin{proof}
    Consider the set
    \[
    \Omega = \{ (c,V)\mid c\in \C,\; V\leq \F_q^n,\; \supp_x(c)\leq V,\; \dim V = i\},
    \]
    which we will count in two ways. First, fix a codeword $c\in \C$ with $\dim \supp_x(c)=j$. The number of $i$-dimensional subspaces $V$ which contain the $j$-dimensional subspace $\supp_x(c)$ is then $\qbinom{n-j}{n-i}_q$. Summing over all possible codewords then yields
    \[
        |\Omega| = \sum^{n}_{j=0} \qbinom{n-j}{n-i}_q A_j(x).
    \]
    Second, fix an $i$-dimensional subspace $V\leq \F_q^n$. As $\supp_x(c)\leq V$ is equivalent to $\supp_x(c) \leq (V^\perp)^\perp$, then $c\in \C(V^\perp,x)$ and by Proposition~\ref{prop:shortenedsize} then $|\C(V^\perp,x)| = q^{m(k-\rho_\C(V^\perp))}$. Since $\dim V = i$ then summing over all such subspaces $V$ is equivalent to summing over all orthogonal complements $W= V^\perp$ of dimension $n-i$, so
    \[ \sum_{\substack{W\leq \F_q^n \\ \dim W = n-i}}q^{m(k- \rho_\C(W))} = \sum_{u=0}^{n}q^{m(k-u)} R_{n-i}^u = |\Omega|.
    \]
    Applying the $q$-binomial inversion formula on the system
    \[
    \sum^{n}_{j=0} \qbinom{n-j}{n-i}_q A_j(x) = \sum_{u=0}^{n}q^{m(k-u)} R_{n-i}^u, \quad i =0,1,\ldots,n,
    \]
    we obtain the result.
\end{proof}

\begin{remark}
As $A_j(x)$ in Theorem~\ref{thm:x-distance} is independent of $x\in \C$ we shall omit it and write $A_j$ to mean $A_j(x)$.
\end{remark}

For linear rank-metric codes, the dual distance distribution describes the exact weight distribution of the orthogonal dual code. Because almost affine codes generally lack an algebraic dual code, their dual distance distribution must be defined formally as the unique algebraic solution to the $q$-ary MacWilliams identity. We demonstrate that this combinatorial distribution is entirely characterized by the dual $q$-matroid.

\begin{definition}
    Let $\C\subseteq \F_{q^m}^n$ be almost affine with $k= \dim \C$. The \emph{dual distance distribution} of $\C$ is the sequence $(B_0,B_1,\ldots,B_n)$ defined as the unique solution to the $q$-ary MacWilliams identity:
    \[
        \sum^{n}_{j=0} \qbinom{n-j}{n-i}_q A_j = q^{m(k+i-n)}\sum^n_{j=0}  \qbinom{n-j}{i}_q B_j, \quad i=0,1,\ldots,n.\]
\end{definition}
\begin{theorem}
    Let $\C\subseteq \F_{q^m}^n$ be almost affine with $k=\dim \C$. For $u,v=0,1,\ldots,n$ denote
    \[
        {R^*}^u_v \coloneqq |\{ V\leq \F_q^n \mid \rho^*_\C(V)=u,\; \dim V = v\}|.
    \]
    Then
    \[
    B_j = \sum^{n}_{i=0}(-1)^{i-j}q^{\binom{i-j}{2}}\qbinom{n-j}{n-i}_q \sum^{n}_{u=0} q^{m(n-k-u)}{R^*}^{u}_{n-i}, \quad j =0,1,\ldots,n.
    \]
\end{theorem}
\begin{proof}
By the proof of Theorem~\ref{thm:x-distance} we have
\begin{align*}
    q^{m(k+i-n)}\sum^{n}_{j=0} \qbinom{n-j}{i}_q B_j = \sum^{n}_{u=0} q^{m(k-u)}{R}_{n-i}^u 
\end{align*}
and dividing by $q^{m(k+i-n)}$ and substituting index $i$ with $n-i$ and $u$ with $v$, we obtain the system
\begin{align}
    \sum^{n}_{j=0}\qbinom{n-j}{n-i}_qB_j = \sum^{n}_{v=0}q^{m(i-v)}{R}_i^v, \quad i =0,1,\ldots,n. \label{eq:dualsystem1}
\end{align}
Recall that as $\rho_\C^*(V)= \dim V + \rho_\C(V^\perp)-k$, so a subspace $V$ of dimension $i$ with rank $v$ in $\M_\C$ corresponds bijectively to its orthogonal complement $V^\perp$ of dimension $n-i$ and dual rank $u=(n-i)+v-k$, so ${R}_i^v = {R^*}_{n-i}^{n-i+v-k}$. Shifting the summation index of \eqref{eq:dualsystem1} to $u=n-i+v-k$, then
\[
    \sum^{n}_{j=0}\qbinom{n-j}{n-i}_qB_j = \sum^{n}_{u=0}q^{m(n-k-u)}{R^*}_{n-i}^u, \quad i =0,1,\ldots,n,
\]
and the result follows by applying $q$-binomial inversion.
\end{proof}

\begin{example}
Let $\C\subseteq \F_{q^m}^n$ be an almost affine rank-metric code with {$n\leq m$} and $2\leq k = \dim \C\leq n-2$. Suppose $\M_\C = \mathcal{U}_{n,k}$. Then 
\[
    R^u_v = \begin{cases}
        \qbinom{n}{v}_q & \text{if }v=u\leq k \text{ or } u=k<v, \\
        0 & \text{otherwise.}
    \end{cases}
\]
By Theorem~\ref{thm:x-distance}, then $A_0=1$ and $A_j=0$ for $1\leq j \leq n-k$. For $t\geq 0$, the non-zero weights evaluate to
\[
    A_{n-k+t} = \qbinom{n}{k-t}_q \sum^{t-1}_{i=0} (-1)^i q^{\binom{i}{2}}\qbinom{n-k+t}{i}_q(q^{m(t-i)}-1).
\]
Since $A_j = 0$ for $1\leq j \leq n-k$, the minimum rank distance of $\C$ is $d=n-k+1$, so $\C$ is MRD. As the distance distribution of $\C$ is non-negative, then
\[
    A_{n-k+2} = \qbinom{n}{k-2}_q\left((q^{2m}-1)- \qbinom{n-k+2}{1}_q(q^m-1) \right) \geq 0,
\]
which implies
\[
    q^{m}\geq \sum^{n-k+1}_{i=1}q^i,
\]
so $m\geq n-k+2$. Dually, as $\M_\C^* = \mathcal{U}_{n,n-k}$ then $B_{k+2}\geq 0$, yielding $m\geq k+2$. Consequently, an almost affine MRD code with $2\leq k \leq n-2$ can exist only if $m\geq \max\{n-k+2,k+2\}$.
\end{example}

Beyond the minimum distance, the hierarchy of subcode supports is captured by the generalized rank weights. For almost affine codes, we evaluate these generalized weights using the nullity function of the dual $q$-matroid, analogous to \cite{10.1109/TIT.2017.2654456}.

\begin{definition}
    Let $\C\subseteq \F_{q^m}^n$ be almost affine. The $i$th \emph{generalized weight} is then
    \[
        d_i(\C) = \min_{V\leq \F_q^n}\{\dim V \mid \rho_\C^*(V) = \dim V - i\},
    \]
    where $i = 1,\ldots,\dim \C$.
\end{definition}

We may equivalently characterize the generalized weights by the dual $q$-matroid or by Proposition~\ref{prop:shortenedsize}, which leads to the following result.

\begin{proposition}
Let $\C\subseteq \F_{q^m}^n$ be almost affine. Let $x\in \C$ be any codeword. Then for every $i=1,\ldots,k$,
\begin{align*}
    d_i(\C) &= \min_{V\leq \F_q^n}\left\{ \dim V \,\middle|\, \rho_\C(V^\perp) = \dim \C - i  \right\} \\
            &= n - \max_{V\leq \F_q^n}\left\{\dim V \,\middle|\, \rho_\C(V) = \dim \C -i\right \} \\
            &= n - \max_{V\leq \F_q^n}\left\{ \dim V \,\middle|\, |\C(V,x)| = q^{mi} \right\}.
\end{align*}
\end{proposition}

\begin{theorem} Let $\C\subseteq \F_{q^m}^n$ be almost affine. Then for every $i=1,\ldots,k$,
\begin{align*}
    d_i(\C) = \min_{\mathcal{D}\subseteq \mathcal{C}} \left\{ \dim \supp \mathcal{D} \,\middle|\, \mathcal{D} \text{ is almost affine and $\dim \mathcal{D} = i$} \right\}.
\end{align*} \label{thm:generalizedweights}
\end{theorem}
\begin{proof}
Let $d_i = d_i(\C)$ and let \[
    e_i = \min_{\mathcal{D}\subseteq \mathcal{C}} \left\{ \dim \supp \mathcal{D} \,\middle|\, \mathcal{D} \text{ is almost affine and $\dim \mathcal{D} = i$} \right\}.
\]

To prove $d_i\leq e_i$, let $\mathcal{D}$ be an almost affine subcode with $\dim (\mathcal{D})=i$ and $\dim \supp \mathcal{D} = e_i$. Let $V= \supp \mathcal{D}$. Fix an element $x\in \mathcal{D}$, then for every $c\in \mathcal{D}$ we have $\supp_x(c) \leq V$. This implies $c\in \mathcal{C}(V^\perp,x)$, so $\mathcal{D}$ is a subset of the flat $\C(V^\perp,x)$. Thus, $i = \dim \mathcal{D} \leq \dim \C(V^\perp,x) = k - \rho_\C(V^\perp)$ by Proposition~\ref{prop:shortenedsize}. This implies the existence of a subspace $U\leq V$ such that $k-\rho_\C(U^\perp)=i$, so $d_i(\C) \leq \dim U \leq \dim V = e_i$.

For the converse, let $V\leq \F_q^n$ such that $k - \rho_\C(V^\perp) = i$ and $\dim V = d_i(\C)$. Fix a codeword $x\in \C$ and consider the flat $\mathcal{D} = \C(V^\perp,x)$ of dimension $\dim \mathcal{D} = k-\rho_\C(V^\perp)=i$. Thus, $\mathcal{D}$ is a valid candidate for the minimum defining $e_i$. For every $c\in \mathcal{D}$ we have $\supp_x(c) \leq V$, so $\supp \mathcal{D}\leq V$, which implies $e_i \leq \dim \supp \mathcal{D} \leq \dim V = d_i$.
\end{proof}

\section{Partial Affine \textit{q}-Geometries} \label{sec:geometry}

The proof of Theorem~\ref{thm:generalizedweights} relies on the incidence structure of flats of the form $\mathcal{C}(Z,x)$. To formalize the geometric behavior of these flats and their parallel classes, we define an abstract incidence geometry. Let $\mathcal{C}\subseteq\mathbb{F}_{q^m}^n$ be an almost affine code with $k=\dim\mathcal{C}$. For $Z\le\mathbb{F}_q^n$ and $x\in\mathcal{C}$, we established in Corollary~\ref{cor:shortened} that $\mathcal{C}(Z,x)$ is an almost affine subcode of $\mathcal{C}$ of dimension $d=k-\rho_{\mathcal{C}}(Z)$. We shall refer to $\mathcal{C}(Z,x)$ as a $d$-dimensional flat of $\mathcal{C}$. Flats of dimension 0, 1, 2 and $k-1$ are points, lines, planes and hyperplanes, respectively. Two flats of the form $L_1 = \C(Z,x)$ and $L_2 = \C(Z,y)$ are \emph{parallel}, denoted $L_1 \| L_2$. Parallel flats are of equal dimension.

\begin{proposition} \label{prop:intersectflats}
Let $\C\subseteq \F_{q^m}^n$ be almost affine, $V,W\leq \F_q^n$, and $x,y\in \C$. If $z\in \C(V,x)\cap \C(W,y)$, then $\C(V,x)\cap \C(W,y) = \C(V+W,z)$.
\end{proposition}
\begin{proof}
An element $c\in \C$ belongs to the intersection $\C(V,x)\cap \C(W,y)$ if and only if $\pi_V(c) = \pi_V(x)$ and $\pi_W(c)= \pi_W(y)$. By assumption we have $\pi_V(x) = \pi_V(z)$ and $\pi_W(y)=\pi_W(z)$. Thus, 
\begin{align*}
    c\in \C(V,x)\cap \C(W,y) &\Leftrightarrow \pi_V(c) = \pi_V(z) \text{ and } \pi_W(c)=\pi_W(z) \\
    &\Leftrightarrow c-z\in V^\perp \otimes_{\F_q}\F_{q^m} \text{ and } c-z\in W^\perp \otimes_{\F_q}\F_{q^m} \\
    &\Leftrightarrow c-z\in (V+W)^\perp \otimes_{\F_q}\F_{q^m} \\
    & \Leftrightarrow \pi_{V+W}(c)=\pi_{V+W}(z) \\
    & \Leftrightarrow c\in \C(V+W,z). \qedhere
\end{align*}
\end{proof}

For a given flat $L = \C(Z,x)$ then the flats of $\C$ parallel to $L$ partition $\C$. The \emph{parallel classes} are $\mathcal{L}(Z) = \{ \C(Z,x) \, |\, x\in \C\}$. Elements of $\mathcal{L}(Z)$ are in one-to-one correspondence with the elements of $\C_Z$. Furthermore, the parallel classes are in one-to-one correspondence with the flats of $\M_\C$. 

For two flats $L_1$ and $L_2$ we define the join $L_1 \vee L_2$ as the smallest flat containing both $L_1$ and $L_2$. While there is no general dimension formula for the join, submodularity provides a strict bound in the case when the two flats intersect non-trivially.

\begin{proposition}
    Let $L_1,L_2$ be flats of an almost affine code $\C\subseteq\F_{q^m}^n$. If $L_1 \cap L_2 \neq \varnothing$, then $
        \dim( L_1 \cap L_2) \geq \dim L_1 + \dim L_2 - \dim (L_1 \vee L_2) \geq \dim L_1 + \dim L_2 - \dim \C.$
\end{proposition}
\begin{proof}
Choose an element $z\in L_1 \cap L_2$ and choose $V,W\leq \F_q^n$ such that $L_1 = \C(V,z)$ and $L_2=\C(W,z)$. By Proposition~\ref{prop:intersectflats} then $L_1 \cap L_2 = \C(V+W,z)$, which has dimension $k-\rho_\C(V+W)$. 

Any flat containing both $L_1$ and $L_2$ must contain $z$, so it must be of the form $\C(U,z)$ for some $U\leq \F_q^n$. For $\C(U,z)$ to contain both flats, then $U\leq V$ and $U\leq W$, so $U\leq V\cap W$. The smallest such flat corresponds to the largest such subspace. The join evaluates exactly to $L_1 \vee L_2 = \C(V\cap W,z)$ with dimension $k-\rho_\C(V\cap W)$. The result then follows by submodularity of $\rho_\C$.
\end{proof}

\begin{remark}
    While the flats of an almost affine code do not globally form a lattice due to disjoint parallel classes, the local structure is highly rigid. For any fixed codeword $z\in \C$, the set of all flats containing $z$ forms a geometric lattice ordered by set inclusion. Because the projection condition $\pi_V(c) = \pi_V(z)$ reverses containment, this local lattice of code flats is anti-isomorphic to the geometric lattice of flats of $\M_\C$.
\end{remark}

\begin{proposition}
    Let $L,L_1,L_2$ be flats of an almost affine code $\C\subseteq \F_{q^m}^n$. If $L_1 \| L_2$ and both $L\cap L_1 \neq \varnothing$ and $L\cap L_2 \neq \varnothing$, then $(L\cap L_1) \| (L\cap L_2)$.
\end{proposition}
\begin{proof}
    Let $U\leq \F_q^n$ and $x\in \C$ such that $L = \C(U,x)$. Since $L_1 \| L_2$, then there exists $V\leq \F_q^n$ and $y_1,y_2\in \C$ such that $L_1 = \C(V,y_1)$ and $L_2 = \C(V,y_2)$. As $L\cap L_1 \neq \varnothing$, there exists $z_1\in L\cap L_1$ such that $L\cap L_1 = \C(U+V,z_1)$. Similarly, $L \cap L_2 = \C(U+V,z_2)$ for some $z_2\in \C$, which proves the claim.
\end{proof}

The following result shows that two hyperplanes are either parallel or intersect in a flat of dimension $k-2$.

\begin{proposition}
    Let $H_1,H_2$ be distinct hyperplanes of an almost affine code $\C$. If $H_1\cap H_2 \neq \varnothing$, then $H_1 \cap H_2$ is a flat of dimension $\dim \C -2$.
\end{proposition}
\begin{proof}
    Let $V,W\leq \F_q^n$ and $x_1,x_2\in \C$ such that $H_1 = \C(V,x_1)$ and $H_2 = \C(W,x_2)$. Then $\rho_\C(V) = \rho_\C(W)=1$. Furthermore, $\rho_\C(V+W) = 2$, since if this rank were $1$, $V$ and $W$ would define the same parallel class, forcing $H_1$ and $H_2$ to be either identical or strictly disjoint, which contradicts the assumption $H_1 \cap H_2 \neq \varnothing$. As $V$ and $W$ are subspaces of rank 1 with distinct closures, they are independent in $\M_\C$, so $\rho_\C(V+W)=2$. For any $z\in H_1\cap H_2$, the intersection evaluates to $\C(V+W,z)$, and this flat has dimension $\dim \C-\rho_\C(V+W) = \dim \C-2$.
\end{proof}

\begin{definition} \label{def:partialgeometry}
A \emph{partial affine $q$-geometry} is a pair $(A,\mathcal{H})$ where $A$ is a set of points of $\F_{q^m}^n$, and $\mathcal{H}$ is a family of subsets of $A$ called \emph{hyperplanes} satisfying the following:
\begin{enumerate}
    \item For every $v\in \F_q^n \setminus\{0\}$ there exists a \emph{parallel class} $\mathcal{H}_v\subseteq \mathcal{H}$ which partitions $A$, such that each $\alpha\in \F_{q^m}$ uniquely indexes a hyperplane $H_{v,\alpha}\in \mathcal{H}_v$.
    \item For every $\F_q$-basis $\{\gamma_1,\ldots,\gamma_n\}$ of $\F_q^n$ and any $\alpha_1,\ldots,\alpha_n\in \F_{q^m}$ the intersection $\bigcap_{i=1}^n H_{\gamma_i,\alpha_i}$ contains at most one point.
    \item If $u,v\in \F_q^n\setminus\{0\}$ such that $u+v\neq 0$, and $P\in H_{u,\alpha}\cap H_{v,\beta}$, then $P\in H_{u+v,\alpha+\beta}$.
    \item For any $D\subseteq \F_q^n\setminus\{0\}$ and any two sets of hyperplanes $\{H_{v,\alpha_v}\}_{v \in D}$ and $\{ H_{v,\beta_v}\}_{v\in D}$, if both $\bigcap_{v\in D} H_{v,\alpha_v}$ and $\bigcap_{v\in D} H_{v,\beta_v}$ are non-empty, then it holds $|\bigcap_{v\in D} H_{v,\alpha_v}| = |\bigcap_{v\in D} H_{v,\beta_v}|$.
\end{enumerate}
\end{definition}

\begin{remark}
    While Property 2 specifies that the intersection of hyperplanes defined by a basis contains at most one point, Property 1 guarantees that every point in $A$ is defined as such an intersection. Because every parallel class $\mathcal{H}_v$ partitions $A$, any arbitrary point $P \in A$ must belong to exactly one hyperplane within each class. Thus, for any chosen $\F_q$-basis $\{\gamma_1, \ldots, \gamma_n\}$ of $\F_q^n$, there exists a unique sequence of scalars $\alpha_1, \ldots, \alpha_n$ such that $\{P\} = \bigcap_{i=1}^n H_{\gamma_i, \alpha_i}$.
\end{remark}

Here and below, if $u=\lambda v$ with $\lambda \in \F_q^\times$, then the corresponding hyperplane classes are identified by the reindexing $\alpha\mapsto \lambda \alpha$; equivalently, $H_{u,\alpha} = H_{v,\alpha \lambda^{-1}}$. In particular, parallelism depends only on the one-dimensional subspace $\langle v\rangle_{\F_q}$.

\begin{definition} \label{def:parallelism}
 Let $(A,\mathcal{H})$ be a partial affine $q$-geometry. Two hyperplanes $H_{u,\alpha}$ and $H_{v,\beta}$ are defined to be \emph{parallel}, denoted $H_{u,\alpha}\| H_{v,\beta}$, if there exists a scalar $\lambda\in  \F_q^\times$ such that $u=\lambda v$.
\end{definition}

One may readily verify that parallelism defines an equivalence relation on the family of hyperplanes $\mathcal{H}$. Furthermore, if $u= \lambda v$, then the parallel classes of $\mathcal{H}_u$ and $\mathcal{H}_v$ define identical partitions on $A$.

\begin{definition}

Let $V\leq \F_q^n$ have a basis $\{\nu_1,\ldots,\nu_d\}$. A flat associated with $V$ is any non-empty intersection of the form $f= \bigcap ^d_{i=1} H_{\nu_i,\alpha_i}$ for some $\alpha_1,\ldots,\alpha_d\in \F_{q^m}$. Two flats $f_1$ and $f_2$ are \emph{parallel} if they are defined by the same subspace $V$, denoted $f_1\| f_2$.
\end{definition}

One may again readily verify that parallelism defines an equivalence relation on the set of all flats of a given dimension $d$. Furthermore, the equivalence class of flats defined by a fixed subspace $V$ partitions the point set $A$ into blocks of equal size.

\begin{proposition}
    Let $\C\subseteq \F_{q^m}^n$ be a simple almost affine code. For all $v\in \F_q^n\setminus \{0\}$ and $\alpha \in \F_{q^m}$ define $H_{v,\alpha} = \{ c\in \C\mid \langle c,v\rangle = \alpha\}$ and $\mathcal{H} = \{H_{v,\alpha} \mid v\in \F_{q}^n\setminus\{0\}, \,\alpha \in \F_{q^m}\}$. Then $(\C,\mathcal{H})$ is a partial affine $q$-geometry.
\end{proposition}
\begin{proof}
    For Property 1 of Definition~\ref{def:partialgeometry}, fix $v\in \F_q^n\setminus \{0\}$. Since $\C$ is simple, $\M_\C$ has no loops, so $\rho_\C(\langle v\rangle _{\F_q})=1$. As $\C$ is almost affine, it follows that $|\pi_{\langle v \rangle}(\C)| = q^m$. The map $c\mapsto \langle c,v\rangle $ has an image in $\F_{q^m}$ of size $q^m$. Hence this map is surjective, and for each $\alpha \in \F_{q^m}$ the set $H_{v,\alpha} = \{c\in \C \mid \langle c,v\rangle = \alpha\}$ is non-empty. Therefore the sets $H_{v,\alpha}$ partition $\C$.
    
    For Property 2, let $\{\gamma_1,\ldots,\gamma_n\}$ be an $\F_q$-basis of $\F_q^n$. The map $T\colon \F_{q^m}^n\rightarrow \F_{q^m}^n$ defined by $T(c) = \left( \langle c,\gamma_1\rangle,\ldots,\langle c,\gamma_n\rangle\right)$ is an $\F_{q^m}$-linear isomorphism, so the system $\langle c,\gamma_i\rangle = \alpha_i$ has at most one solution in $\F_{q^m}^n$, and therefore at most one solution in $\C$.

    Property 3 follows immediately by bilinearity of the inner product. 

    For Property 4, let $D\subseteq \F_q^n\setminus\{0\}$, let $V= \langle D\rangle_{\F_q}$, and let $\{\nu_1,\ldots,\nu_d\}$ be an $\F_q$-basis of $V$ contained in $D$. For any choice of $\alpha_1,\ldots,\alpha_d \in \F_{q^m}$, the intersection $I_\alpha = \bigcap ^d _{i=1}H_{\nu_i,\alpha_i}$ is exactly the set of codewords $c\in \C$ satisfying $\langle c,\nu_i\rangle = \alpha_i$. By linearity, these basis constraints dictate the inner products for all remaining vectors in $D$. Consequently, a non-empty intersection over all of $D$ is equal to $I_\alpha$. If $I_\alpha$ is non-empty, then $I_\alpha = \C(V,x)$ for any $x\in I_\alpha$. By Proposition~\ref{prop:shortenedsize}, $|I_\alpha| = q^{m(\dim \C - \rho_\C(V))}$, which depends only on $V$, not on the choice of $\alpha$. This proves Property 4.

    Finally, the requirement that $\mathcal{C}$ is simple is strictly necessary. If $\mathcal{C}$ were not simple, there would exist linearly independent vectors $u, v \in \mathbb{F}_q^n$ such that $\rho_{\mathcal{C}}(\langle u, v \rangle) = 1$. This forces $H_{u, \alpha} = H_{v, \beta}$ for some scalars $\alpha, \beta \in \mathbb{F}_{q^m}$, meaning $u$ and $v$ define identical partitions on $\mathcal{C}$. This contradicts Definition~\ref{def:parallelism} being an equivalence relation, as it assigns the same set of codewords to two distinct, non-parallel classes.
\end{proof}

\begin{proposition}
    Let $(A,\mathcal{H})$ be a partial affine $q$-geometry. Fix an $\F_q$-basis $\{\gamma_1,\ldots,\gamma_n\}$ of $\F_q^n$. For each $P\in A$, let $\alpha_i\in \F_{q^m}$ be the unique scalar such that $P \in H_{\gamma_i,\alpha_i}$ for $i=1,\ldots,n$. Let $\phi_E \colon A \rightarrow \F_{q^m}^n$ such that $\phi_E(P) = (\alpha_1,\ldots,\alpha_n)$. Then $\phi_E$ is injective and $\C = \{\phi_E(P) \mid P \in A\}$ is a simple almost affine code.
\end{proposition}
\begin{proof}
    To show injectivity, suppose $\phi_E(P) = \phi_E(Q) = (\alpha_1, \ldots, \alpha_n)$. This implies that both $P$ and $Q$ lie in the intersection $\bigcap_{i=1}^n H_{\gamma_i, \alpha_i}$. By Property 2 of Definition~\ref{def:partialgeometry}, this intersection contains at most one point, forcing $P = Q$. Thus, $\phi_E$ is injective.

    Next, we must prove that the assignment of coordinates respects $\F_q$-linearity across all directions, not just the basis $E$. Let $c = \phi_E(P) = (\alpha_1, \ldots, \alpha_n) \in \mathcal{C}$ and let $v = \sum_{i=1}^n \lambda_i \gamma_i \in \F_q^n \setminus \{0\}$. Since $P \in \bigcap_{i=1}^n H_{\gamma_i, \alpha_i}$, repeated application of the additivity and scaling rules of Property 3 guarantees that $P \in H_{v, \beta}$, where $\beta = \sum_{i=1}^n \lambda_i \alpha_i$. Notice that $\beta$ evaluates to $\langle c, v \rangle$. Therefore, for any codeword $c = \phi_E(P)$ and any $v \in \F_q^n \setminus \{0\}$, the inner product $\langle c, v \rangle$ retrieves the label of the hyperplane $H_{v, \beta} \in \mathcal{H}_v$ containing $P$.

    To verify that $\mathcal{C}$ is almost affine, let $V \le \F_q^n$ be an arbitrary subspace and choose an $\F_q$-basis $V = \langle \nu_1, \ldots, \nu_d \rangle_{\F_q}$. For $i=1,\ldots,d$, let $V_i = \langle \nu_1, \ldots, \nu_i \rangle_{\F_q}$ and let $\mathcal{C}_i = \pi_{V_i}(\mathcal{C})$. For the base case $V_1$, Property 1 yields $|\mathcal{C}_1| = q^m$. When passing from $V_{i-1}$ to $V_i$, the fibers of $\pi_{V_i}$ are formed by evaluating the fibers of $\pi_{V_{i-1}}$ on $\nu_i$. By our linearity argument above, this evaluation corresponds to intersecting the existing fibers with the hyperplanes of the parallel class $\mathcal{H}_{\nu_i}$. Property 4 guarantees that all non-empty such intersections have the same size, and Property 3 shows that the pattern of non-empty intersections depends only on the subspace $V_i$, not on the chosen basis. Consequently, each step from $\mathcal{C}_{i-1}$ to $\mathcal{C}_i$ multiplies the number of non-empty fibers by $1$ or $q^m$. In particular, $|\mathcal{C}_i|$ is always a power of $q^m$, so $\mathcal{C}$ is an almost affine rank-metric code.

    Finally, we show that $\mathcal{M}_{\mathcal{C}}$ is simple. Since $|\pi_{\langle v \rangle_{\F_q}}(\mathcal{C})| = q^m$ for every non-zero $v \in \F_q^n$, the code contains no loops. To rule out parallel elements, suppose $u, v \in \F_q^n$ are linearly independent. If their joint projection collapsed to rank 1, meaning $|\pi_{\langle u, v \rangle_{\F_q}}(\mathcal{C})| = q^m$, then the evaluations on $u$ and $v$ would depend on one another. Geometrically, this would force $u$ and $v$ to determine the exact same partition of $A$, contradicting the equivalence relation of parallelism in Definition~\ref{def:parallelism}. Hence, $|\pi_{\langle u, v \rangle_{\F_q}}(\mathcal{C})| = q^{2m}$, meaning there are no 2-dimensional circuits. Therefore, $\mathcal{M}_{\mathcal{C}}$ is simple.
\end{proof}

\begin{corollary}
    There is a one-to-one correspondence between the isomorphism classes of partial affine $q$-geometries and the isomorphism classes of simple almost affine rank-metric codes. \label{cor:equivalencegeometry}
\end{corollary}

An explicit construction of a simple, strictly almost affine rank-metric code can be derived from Additive Generalized Twisted Gabidulin (AGTG) codes, introduced in \cite{additiverankmetriccodes}.

\begin{example}
Let $q_0$ be a prime power, and let $n,k,s,u,h\in \mathbb{N}$ be positive integers satisfying $\gcd(n,s)=1$, $k<n$, and denote $q=q_0^u$. Let $\eta\in \F_{q^n}$ such that $N_{q^{sn}/q_0^s}(\eta) \neq (-1)^{nku}$, where $N_{q^{sn}/q_0^s}$ denotes the field norm over $\F_{q^{sn}}/\F_{q_0^s}$. Then 
\[
    \mathcal{A}_{k,s,q_0}(\eta,h) = \left\{ \alpha_0 x +\alpha_1 x^{q^s} + \ldots + \alpha_{k-1}x^{q^{s(k-1)}} + \eta \alpha_0^{q^h_0}x^{q^{sk}} \;\middle|\; \alpha_0,\ldots,\alpha_{k-1}\in \F_{q^n} \right\}
\]
is an Additive Generalized Twisted Gabidulin code. By evaluating these polynomials on a fixed $\F_q$-basis of $\F_{q^n}$, the code can be represented as a set of $n\times n$ matrices over $\F_q$, which natively maps to a vector rank-metric code $\C\subseteq \F_{q^n}^n$. By construction, AGTG codes are always $\F_{q_0}$-linear, but they are strictly not $\F_q$-linear whenever $u\nmid h$.

We instantiate this family with the parameters $q_0=2$, $u=2$ (so $q=4$), $n=5$, $k=2$, $s=2$, and $h=1$, alongside a valid non-zero twist parameter $\eta\in \F_{4^5}$. This yields an $\F_2$-linear code $\C\subseteq \F_{1024}^5$. By the AGTG construction, this is an MRD code with minimum rank distance $d=4$. This code is therefore simple as the code contains no codewords of rank 1 or 2. Furthermore, it is almost affine as it is MRD, and it is strictly almost affine as it is not $\F_{1024}$-linear. Thus, it induces a proper partial affine $q$-geometry. \label{exmp:AGTG}
\end{example}

To highlight the difficulty of finding codes that are simultaneously simple and strictly almost affine, we now contrast the AGTG family with classical geometric constructions. Specifically, the explicit code presented earlier in Example~\ref{exmp:almostaffine} is a punctured instance of a broader class of codes derived from finite translation planes. While these classical structures naturally yield strictly almost affine codes, we will demonstrate that their rank-metric adaptations inherently fail to be simple.

Constructing codes from these geometric incidence structures is a classical technique. It is a well-established geometric property that affine translation planes can be coordinatized by proper finite semifields, yielding line equations determined by non-associative bilinear operations \cite{projectiveplanes,Assmus_Key_1992}. Furthermore, evaluating the points of an affine plane across its parallel classes naturally produces a net. It is a standard combinatorial result that such nets are equivalent to orthogonal arrays, which yield MDS codes in the Hamming metric. In \cite{AlmostAffineCodes} it was then shown that these structures yield almost affine codes in the Hamming metric.

In the following construction, we generalize the mechanism behind Example~\ref{exmp:almostaffine}. In particular, Example~\ref{exmp:almostaffine} is a puncturing of a code induced by Theorem~\ref{thm:construction2dim}. We explicitly apply this classical coordinate evaluation map to an arbitrary semifield translation plane and lift the resulting non-$\F_{q^m}$-linear additive code into the rank metric, and we demonstrate that this code is also almost affine as a rank-metric code.

Recall that a \emph{finite semifield} $(\S,+,\circ)$ is a finite non-associative division ring with multiplicative identity element. By definition, $\S$ contains a field $\F_q$ in its center, meaning the elements of $\F_q$ commute and associate with all other elements in $\S$. As scalar multiplication can be restricted to this central field, $\S$ forms a finite-dimensional vector space over $\F_q$. If we denote the dimension of this vector space by $m$, the order of the semifield is $|\S| = q^m$. Consequently, we can identify the underlying additive group and set of $\S$ with the field $\F_{q^m}$. Furthermore, because the scalars $\lambda \in \F_q$ associate and commute with all elements, we have $\lambda (x\circ y) = (\lambda x)\circ y = x\circ (\lambda y)$ for all $x,y\in \S$. Thus, $\circ$ is an $\F_q$-bilinear map. A semifield is called \emph{proper} if its multiplication operation is strictly non-associative. 

\begin{theorem}
Let $\S$ be a finite proper semifield of order $q^m$, and let $$\C_\S = \{ (x,(y-x\circ a))_{a\in \S}\mid (x,y)\in \S^2\}.$$ Then $\C_\S \subseteq \F_{q^m}^n$ is an $\F_q$-linear, strictly almost affine rank-metric code with parameters $n=q^m+1$, $\dim \C_\S = 2$, and $d(\C_\S)=1$. \label{thm:construction2dim}
\end{theorem}
\begin{proof}
First, we determine the parameters of the code. The length of the code clearly holds as we identify $\mathcal{S}$ with $\mathbb{F}_{q^m}$, and we identify the coordinates with $\{\infty\} \cup \mathcal{S}$. The size of the code is $|\mathcal{S}|^2 = (q^m)^2$, which implies $\dim_{\mathbb{F}_{q^m}} \mathcal{C}_{\mathcal{S}} = 2$. Consider the encoding map $c\colon \mathcal{S}^2 \rightarrow \mathcal{S}\times \mathcal{S}^{q^m}$ given by $(x,y)\mapsto (x,(y-x\circ a)_{a\in \mathcal{S}})$. As the semifield multiplication is $\mathbb{F}_q$-bilinear, the encoding map $c(x,y)$ is an $\mathbb{F}_q$-linear transformation. Consequently, $\mathcal{C}_{\mathcal{S}}$ is an $\mathbb{F}_q$-linear subspace of $\mathbb{F}_{q^m}^n$. To determine the minimum distance, let $x=0$ and $y\in \mathcal{S}^\times$. The resulting codeword is $c(0,y) = (0,(y-0)_{a\in \mathcal{S}})=(0,y,\ldots,y)$. The rank weight of this codeword is clearly $1$, and as $\mathcal{C}_{\mathcal{S}}$ contains the zero codeword, the minimum rank distance of $\mathcal{C}_{\mathcal{S}}$ is 1.

To prove that $\mathcal{C}_{\mathcal{S}}$ is an almost affine rank-metric code, consider a subspace $V\leq \mathbb{F}_q^n$. Then $|\pi_V(\mathcal{C}_{\mathcal{S}})|= \frac{q^{2m}}{|\ker \pi_V|}$, and the kernel consists of codewords satisfying $\langle c(x,y),v\rangle =0$ for all $v\in V$. For a vector $v=(v_{\infty},(v_a)_{a\in \mathcal{S}}) \in \mathbb{F}_q^n$, this condition implies
\begin{align}
    v_\infty x + \sum_{a\in \mathcal{S}} v_a(y-x\circ a) =0. \label{eq:orthocode1}
\end{align}
As $v_a\in \mathbb{F}_q$, and the semifield is an $\mathbb{F}_q$-algebra, the scalars commute with $\circ$, so $v_a(x\circ a ) = x\circ (v_a a)$. Now, denote $A_v = \sum_{a\in \mathcal{S}} v_a \in \mathbb{F}_q$ and $B_v = v_\infty - \sum_{a\in \mathcal{S}} v_a a \in \mathcal{S}$. Substituting this into \eqref{eq:orthocode1} we obtain
\begin{align}
    A_v y + x\circ B_v = 0. \label{eq:orthocode2}
\end{align}

We consider two cases. For the first case, suppose $A_v =0$ for all $v\in V$. Then \eqref{eq:orthocode2} reduces to $x\circ B_v =0$. If $B_v = 0$ for all $v\in V$, then $\ker \pi_V = \mathcal{S}^2$, so $|\pi_V(\mathcal{C}_{\mathcal{S}})| = 1$. If $B_v \neq 0$ for some $v\in V$, then $x=0$ as $\mathcal{S}$ has no zero divisors. This implies that $y$ is a free variable, so the kernel is $\{0\} \times \mathcal{S}$, which yields $|\pi_V(\mathcal{C}_{\mathcal{S}})|= q^m$.

For the second case, suppose $A_w \neq 0$ for some $w\in V$. Then $y= -x\circ (B_w A_w^{-1})$. For all other $v\in V\setminus \{w\}$, \eqref{eq:orthocode2} then reduces to
\[
    0 = A_v(-x\circ (B_wA_w^{-1}))+ x\circ B_v = x\circ (B_v - A_v B_w A_w^{-1}).
\]
Now, if $B_v - A_v B_w A_w^{-1}=0$ for all $v\in V\setminus \{w\}$, then $x$ is a free variable and determines $y$. The kernel is therefore isomorphic to $\mathcal{S}$, yielding $|\pi_V(\mathcal{C}_{\mathcal{S}})|= q^m$. If this term is non-zero for any $v\in V\setminus \{w\}$, it forces $x=0$ and thus $y=0$. The kernel is trivial, so $|\pi_V(\mathcal{C}_{\mathcal{S}})| = q^{2m}$. Because all projections yield sizes of $1$, $0q^m$, or $q^{2m}$, $\mathcal{C}_{\mathcal{S}}$ is almost affine.

Finally, we prove that $\mathcal{C}_{\mathcal{S}}$ is strictly almost affine by showing it is not an affine space over $\mathbb{F}_{q^m}$. Suppose for contradiction that it is. Because encoding $(0,0)$ yields the zero vector, $\mathcal{C}_{\mathcal{S}}$ contains the origin, meaning any such affine space must be a linear vector space over the field $\mathbb{F}_{q^m}$. Let $\cdot$ denote the standard field multiplication of $\mathbb{F}_{q^m}$. For $\mathcal{C}_{\mathcal{S}}$ to be an $\mathbb{F}_{q^m}$-linear subspace, we must have $\lambda \cdot c \in \mathcal{C}_{\mathcal{S}}$ for any $c \in \mathcal{C}_{\mathcal{S}}$ and any scalar $\lambda \in \mathbb{F}_{q^m}$.

Consider the valid codeword $c(x,0) = (x, (-x \circ b)_{b \in \mathcal{S}})$. Multiplying this codeword by $\lambda$ using the field scalar multiplication yields the vector $v = (\lambda \cdot x, (\lambda \cdot (-x \circ b))_{b \in \mathcal{S}})$. If $v$ belongs to the code, it must equal $c(\hat{x}, \hat{y})$ for some $\hat{x}, \hat{y} \in \mathcal{S}$. Matching the first coordinate dictates $\hat{x} = \lambda \cdot x$. Matching the coordinate at $b=0$ yields $\lambda \cdot (-x \circ 0) = \hat{y} - (\lambda \cdot x) \circ 0$, which forces $\hat{y} = 0$. Therefore, if $v$ is in the code, it must be $c(\lambda \cdot x, 0)$. However, examining an arbitrary coordinate $a \in \mathcal{S}^{\times}$, the vector $v$ contains $-\lambda \cdot (x \circ a)$, whereas $c(\lambda \cdot x, 0)$ contains $-(\lambda \cdot x) \circ a$. Because $\mathcal{S}$ is a proper semifield, its multiplication $\circ$ is strictly distinct from the field multiplication $\cdot$. Thus, there exist elements where $\lambda \cdot (x \circ a) \neq (\lambda \cdot x) \circ a$, meaning the coordinates do not match. The scaled vector is not in the code, proving $\mathcal{C}_{\mathcal{S}}$ is not $\mathbb{F}_{q^m}$-linear.
\end{proof}
\begin{remark}
The behavior of $\mathcal{C}_{\mathcal{S}}$ depends strongly on the choice of metric, reflecting a contrast between classical combinatorial design and rank-metric geometry. In the Hamming metric, the parallel classes of the translation plane ensure that $\mathcal{C}_{\mathcal{S}}$ forms an orthogonal array of index one, and hence an MDS code. In the rank metric, however, the non-zero parallel classes collapse to identical coordinates over the base field, so the code has rank distance equal to one. Moreover, $\mathcal{C}_{\mathcal{S}}$ is an $\mathbb{F}_q$-linear subspace of $\mathbb{F}_{q^m}$ without being $\mathbb{F}_{q^m}$-linear.
\end{remark}

We now generalize the construction to codes of higher dimensions. To preserve the almost affine property while avoiding the dimension collapse caused by non-associativity in higher dimensions, we restrict the evaluation map to a one-dimensional subspace over the central field $\mathbb{F}_q$.

\begin{theorem} Let $\S$ be a finite proper semifield of order $q^m$, and let $k\geq 2$. Choose a point $p =(p_1,\ldots,p_{k-1})\in \S^{k-1}$ such that there exist elements $\gamma,x\in \S$ satisfying $\gamma \circ (x\circ p_j) \neq (\gamma\circ x)\circ p_j$ for some coordinate index $j$. Define
\[
    \C_{\S,k}(p) = \left\{ \left(x_1,\ldots,x_{k-1},\left( y - \sum^{k-1}_{i=1}x_i \circ (\lambda p_i) \right)_{\lambda \in \F_q} \right) \middle| (x_1,\ldots,x_{k-1},y)\in \S^k \right\}.
\]
Then $\C_{\S,k}(p)$ is an $\F_q$-linear, strictly almost affine rank-metric code with parameters $n=(k-1)+q$, $\dim \C_{\S,k}(p) = k$, and $d(\C_{\S,k}(p)) = 1$. \label{thm:constructionkdim}
\end{theorem}
\begin{proof}
The parameters and $\F_q$-linearity follow directly by the same arguments as in Theorem~\ref{thm:construction2dim}. To prove $\mathcal{C}_{\mathcal{S},k}(p)$ is not $\mathbb{F}_{q^m}$-linear, suppose for contradiction that it is. Let $\cdot$ denote the standard field multiplication of $\F_{q^m}$. Consider the codeword $c$ generated by $(x_1, \dots, x_{k-1}, y)$. Matching the first $k-1$ coordinates and the evaluation at $\lambda = 0$ forces $\gamma \cdot c$ to equal the encoding of $(\gamma \cdot x_1, \dots, \gamma \cdot x_{k-1}, \gamma \cdot y)$. Equating the remaining coordinates therefore requires
    $$\sum_{i=1}^{k-1} \gamma \cdot (x_i \circ (\lambda p_i)) = \sum_{i=1}^{k-1} (\gamma \cdot x_i) \circ (\lambda p_i)$$
for all $\lambda \in \mathbb{F}_q$. By hypothesis, there exist an index $j$ and elements $\gamma, x \in \mathcal{S}$ such that $\gamma \circ (x \circ p_j) \neq (\gamma \circ x) \circ p_j$. Setting $x_j = x$, $x_i = 0$ for $i \neq j$, and $y = 0$, the condition at $\lambda = 1$ reduces to $\gamma \cdot (x \circ p_j) = (\gamma \cdot x) \circ p_j$. Because the proper semifield multiplication $\circ$ is strictly distinct from the field multiplication $\cdot$, this equality fails. The scaled vector is therefore not in the code, providing the contradiction.

To prove the almost affine property, consider a projection $\pi_V$ defined by an $\F_q$-subspace $V\leq \F_{q}^n$. A vector $v\in V$ is indexed as $v= (v_{\infty_1},\ldots,v_{\infty_{k-1}},(v_\lambda)_{\lambda \in \F_q})$. The condition $\langle c,v\rangle =0$ then expands to
\[
    \sum^{k-1}_{i=1} v_{\infty_i} x_i + \sum_{\lambda \in \F_q} v_\lambda \left(y - \sum^{k-1}_{i=1} x_i \circ (\lambda p_i)\right) =0.
\]
Because the scalars $\lambda \in \F_q$ associate and commute with all elements in $\S$, we have $x_i \circ (\lambda p_i) = \lambda (x_i \circ p_i)$. Regrouping terms yields
\begin{align}
    A_v y + \sum^{k-1}_{i=1} x_i \circ (v_{\infty_i} 1_\S - C_v p_i) = 0, \label{eq:rank1_kernel}
\end{align}
where $A_v = \sum_\lambda v_\lambda \in \F_q$ and $C_v = \sum_\lambda \lambda v_\lambda \in \F_q$. Define the map $C\colon V\rightarrow \F_q$ by $C(v) = C_v$. Because the coordinate summations are linear over the base field, $C$ is a well-defined linear function on $V$. Let $V_0 = \ker C$. Thus, $\dim_{\F_q} V_0 \geq \dim_{\F_q} V -1$. For any $v\in V_0$ we have $C_v =0$, so \eqref{eq:rank1_kernel} reduces to
\begin{align} \label{eq:v0_system}
    A_v y + \sum^{k-1}_{i=1} v_{\infty_i} x_i=0.
\end{align}

Let $W\subseteq \S^k$ denote the space of solutions $(x_1,\ldots,x_{k-1},y)$ of \eqref{eq:v0_system} for all $v\in V_0$. We then claim that $W$ is a left vector space over $\S$. Closure under vector addition follows immediately from the distributivity of the semifield multiplication over addition. To verify closure under left-multiplication, suppose $u=(x_1,\ldots,x_{k-1},y)\in W$, and let $s\in \S$. Substituting $s\circ u$ into the linear form then yields for any $v\in V_0$
\begin{align*}
    A_v(s\circ y) + \sum^{k-1}_{i=1}v_{\infty_i} (s\circ x_i) &= s\circ (A_v y) + \sum^{k-1}_{i=1} s\circ (v_{\infty_i} x_i) \\
    &= s\circ \left( A_v y + \sum^{k-1}_{i=1} v_{\infty_i} x_i\right) \\ &=0
\end{align*}
Thus, $W$ is a left vector space over $\S$, so its cardinality is a power of $|\S|=q^m$. 

Now, if $V = V_0$, then $\ker \pi_V = W$, and we are done. If $V \neq V_0$, the kernel $V_0$ has codimension 1 in $V$. We can therefore choose a single vector $w \in V \setminus V_0$ such that $V = V_0 \oplus \langle w \rangle_{\mathbb{F}_q}$. The total kernel of the projection by $V$ is the intersection of the $\S$-vector space $W$ with the single additional constraint imposed by $w$, which is
\[
    A_w y + \sum^{k-1}_{i=1} x_i \circ (w_{\infty_i} 1_\S - C_w p_i)=0.
\]
As $W$ is a left vector space over $\S$ and $\S$ is a division ring, adding this single linear constraint to $W$ will either be redundant and contain all of $W$, or it will uniquely determine one variable over $\S$ in terms of the others. This removes exactly one degree of freedom over $\S$, reducing the total number of solutions by a factor of $q^m$. In either case, the size of the final kernel remains a power of $q^m$.
\end{proof}

\begin{remark}
While the codes $\C_{\S,k}(p)$ constructed in Theorem~\ref{thm:constructionkdim} (as well as $\C_\S$ in Theorem~\ref{thm:construction2dim}) are strictly almost affine, their induced $q$-matroids are not simple. They are indeed loopless, however, the projection constraints allow for $2$-dimensional subspaces with projections of size $q^m$. By Corollary~\ref{cor:equivalencegeometry}, they do not form partial affine $q$-geometries.
\end{remark}

The explicit constructions provided here illustrate the difficulty of obtaining simple, strictly almost affine rank-metric codes across arbitrary parameter regimes. The codes $\C_\S$ and $\C_{\S,k}(p)$ derived from finite proper semifields exist whenever proper semifields with the requisite non-associativity properties can be constructed, but they are restricted to a minimum rank distance of 1 and induce non-simple $q$-matroids. Conversely, the AGTG construction yields a simple, strictly almost affine MRD code with a larger minimum distance, but this specific behavior relies on the parameter constraint $n=m$. Constructing simple, strictly almost affine rank-metric codes with minimum rank distance $d>1$ for lengths $n\neq m$ therefore seems like an interesting open problem.

\bibliographystyle{IEEEtran}
\bibliography{biblio}

\end{document}